\documentclass[11pt]{amsart}
\usepackage{amsmath,amssymb,latexsym,soul,cite,mathrsfs}

\usepackage{color,enumitem,graphicx}
\usepackage[colorlinks=true,urlcolor=blue,
citecolor=red,linkcolor=blue,linktocpage,pdfpagelabels,
bookmarksnumbered,bookmarksopen]{hyperref}
\usepackage[english]{babel}

\usepackage[left=2.9cm,right=2.9cm,top=2.8cm,bottom=2.8cm]{geometry}
\usepackage[hyperpageref]{backref}

\usepackage[colorinlistoftodos]{todonotes}
\makeatletter
\providecommand\@dotsep{5}
\def\listtodoname{List of Todos}
\def\listoftodos{\@starttoc{tdo}\listtodoname}
\makeatother

\numberwithin{equation}{section}

\def\cal{\mathcal}

\newtheorem{theorem}{Theorem}[section]
\newtheorem{proposition}[theorem]{Proposition}
\newtheorem{lemma}[theorem]{Lemma}
\newtheorem{corollary}[theorem]{Corollary}

\newtheorem{remark}{Remark}

\DeclareMathOperator{\cat}{cat}

\title[Fractional Laplacian in expanded domains]
{Multiplicity results for the fractional laplacian in expanded domains}

\author[G. M. Figueiredo]{Giovany M. Figueiredo}
\author[M. Pimenta]{Marcos T. O. Pimenta}
\author[G. Siciliano]{Gaetano Siciliano}

\address[G. M. Figueiredo]{\newline\indent Faculdade de Matem\'atica
\newline\indent 
Universidade Federal do Par\'a
\newline\indent
66075-110, Bel\'em - PA, Brazil}
\email{\href{mailto:giovany@ufpa.br}{giovany@ufpa.br}}
\address[M. T. O. Pimenta]{\newline\indent Departamento de Matem\'atica e Computa\c{c}\~ao
\newline\indent 
Faculdade de Ci\^encias e Tecnologia
\newline\indent
UNESP - Universidade Estadual Paulista
\newline\indent
19060-900, Presidente Prudente - SP - Brazil}
\email{\href{mailto:pimenta@fct.unesp.br}{pimenta@fct.unesp.br}}
\address[G. Siciliano]{\newline\indent Departamento de Matem\'atica
\newline\indent 
Instituto de Matem\'atica e Estat\'istica
\newline\indent 
 Universidade de S\~ao Paulo 
\newline\indent 
Rua do Mat\~ao 1010,  05508-090 S\~ao Paulo, SP, Brazil }
\email{\href{mailto:sicilian@ime.usp.br}{sicilian@ime.usp.br}}

\thanks{Giovany M. Figueiredo was partially
supported by  CNPq, Brazil. Marcos T.O. Pimenta was supported by Fapesp and CNPq, Brazil. Gaetano Siciliano  was partially supported by
Fapesp and CNPq, Brazil. }

\begin{document}

\maketitle
\begin{abstract}
\noindent In this paper we establish the multiplicity of nontrivial
weak solutions for the problem $(-\Delta)^{\alpha} u
+u= h(u)$ in $\Omega_{\lambda}$,\ $u=0$ on
$\partial\Omega_{\lambda}$, where
$\Omega_{\lambda}=\lambda\Omega$, $\Omega$ is a smooth and bounded domain in
$\mathbb{R}^N, N>2\alpha$, $\lambda$ is a positive parameter, $\alpha \in (0,1)$, $(-\Delta)^{\alpha}$ is the fractional Laplacian
and the nonlinear  term $h(u)$ has a
subcritical growth. We use minimax methods, the Ljusternick-Schnirelmann and
Morse theories to get multiplicity result depending on the topology of $\Omega$.

\end{abstract}
\maketitle

\section{Introduction}

This paper is concerned with the following
problem
\begin{equation}\label{principal}
\left\{
\begin{array}{ll}
(-\Delta)^{\alpha} u +u= h(u)  & \mbox{ in }  \Omega_{\lambda},\\
u>0 \\
 u=0 & \mbox{ on }  \partial\Omega_{\lambda},
\end{array}
\right.
\end{equation}
where $\Omega_{\lambda}=\lambda\Omega$, $\Omega$ is a smooth and bounded domain in $\mathbb{R}^N, N>2\alpha$,
$\lambda$
is a positive parameter, $\alpha \in (0,1)$,  $(-\Delta)^{\alpha}$ is the fractional Laplacian operator,
whose definition will be briefly recalled in the next section, and $h$ satisfies suitable assumptions.

\medskip

We are motivated in studying an equation involving the fractional Laplacian due to  the great attention which 
has been given in these last years to problems involving  fractional operators, both in $\mathbb R^{N}$ and in bounded domains.
Indeed these problems appear in many areas such as physics, economy, finance, optimization, obstacle problems, fractional diffusion and probabilistic.
In particular, from a probability point of view, the fractional Laplacian is the infinitesimal generator of a L\'evy
process, see e.g. \cite{Bertoin}. We also recall that a fractional Schr\"odinger equation has been derived by Laskin
in the framework of the Fractional Quantum Mechanics. More information and applications are contained in some references such as
\cite{Apple,Cont,L1,MK,Vla}.

On the other hand, in a beautiful series of papers, Benci, Cerami and Passaseo (see \cite{BC1,BC3,BCP})
investigate the existence and multiplicity of positive solutions
for  equations of type $-\Delta u+\lambda u=u^{p-1}$ 
or $-\varepsilon\Delta u+u=f(u)$ in a bounded domain $\Omega$ with Dirichlet boundary conditions.
 In particular they develop a tool
which allows to estimate the number of solutions depending on the ``shape''
of the domain (or of suitable ``nearby'' domains),  whenever the parameters $\lambda,\varepsilon$ or $p$
tend to a suitable limit value.
  They use variational methods,
and introduce  suitable maps which permit to see  ``a photography''
of $\Omega$ in a certain sublevel set of the energy functional related to the equation. Then the Ljusternick-Schnirlemann
and Morse theory, based on the properties of the {\sl category} and some Morse relations,
 are used in order to obtain the existence of multiple solutions.
Later on, these general ideas are  successfully applied also in other contests, 
such as the ``zero mass'' case in \cite{Vis},  Klein-Gordon and Schr\"odinger-Poisson type equations in \cite{GhiGri, GhiMic, Sicilia},
$p-$laplacian equations in \cite{Alves1,Alvesgio1,CV, CV2, CV3, CVV},  quasilinear equations in \cite{AUG,CGU}, 
 fractional Schr\"odinger equation in $\mathbb R^{N}$ with a potential in \cite{GG},
 problems involving magnetic fields in expanding domains in \cite{AFF,CSR}, among many others.

\medskip

The aim of this paper is to show existence and
 multiplicity results of  solutions for the fractional scalar field equation \eqref{principal}
 in the expanding domain $\Omega_{\lambda}.$ 
We obtain the same type of results  of the papers cited above:
 roughly speaking, 
 for $\lambda$ large enough the number
of positive solutions is bounded below by topological invariants related to $\Omega_{\lambda}$.

More precisely,
let us assume that $h:\mathbb{R}\rightarrow\mathbb{R}$ is a $C^1-$function verifying the following conditions:\smallskip
\begin{enumerate}[label=(H\arabic*),ref=H\arabic*,start=0]
\item\label{H_{0}}$h(s)=0$ for $s\leq0$; \smallskip
\item\label{H_{1}} $h(s)=o(|s|)$ at the origin;  \smallskip
\item\label{H_{2}} $ \lim_{|s|\rightarrow \infty}h(s)/|s|^{q-1}=0$ for some $q\in (2,2^{*}_\alpha)$ where
$2^*_\alpha = 2N/(N-2\alpha)$;  \smallskip
\item\label{H_{3}} there exists $\theta > 2$ such that $0<\theta
H(s)\leq sh(s)$ for all $s>0$, where $H(s)=\int_{0}^{s}h(t)\,dt$;  \smallskip
\item\label{H_{4}} the function $s\mapsto h(s)/s$ is
increasing for $s>0$. \smallskip
\end{enumerate}
The typical function satisfying the above conditions
is $h(s)=s^{\mu}$ for $s\geq 0$, with
$1<\mu<q-1$, and
$h(s)=0$ for $s<0$.\\

Our main results are the following.

\begin{theorem}
\label{Main2} Suppose that  \eqref{H_{0}}-\eqref{H_{4}} hold. Then there exists
$\lambda^{*}>0$ such that for $\lambda\geq\lambda^{*}$, problem
(\ref{principal}) has at least $cat \,\Omega_{\lambda}$ 
weak solutions.
\end{theorem}

For $Y\subset X$, we are denoting with  $\cat_{X}(Y)$  the Ljusternick-Schnirelmann category
of $X$ in $Y$, i.e.
the least number of closed and contractible sets in $X$
which cover $Y$. When $X=Y$ we just write $\cat(X)$.

As usual, we get one more solution if the domain $\Omega_{\lambda}$ is not contractible, i.e.

\begin{theorem}\label{+1}
Beside the assumptions of the previous theorem, assume that $cat\,\Omega_{\lambda}>1$.
Then there exists
$\lambda^{*}>0$ such that for $\lambda\geq\lambda^{*}$, problem
(\ref{principal}) has at least $cat \,\Omega_{\lambda}+1$ 
weak solutions.
\end{theorem}

\medskip

If we replace \eqref{H_{1}} and \eqref{H_{2}} with  slightly stronger conditions in order
to deal with the second variation of the energy functional
associated to problem \eqref{principal},
we can get
a better result by using the Morse Theory. To this aim, let \smallskip
\begin{enumerate}[label=(H\arabic*'),ref=H\arabic*',start=1]
\item\label{H_{1}'} $h'(s)=o(|s|)$ at the origin; \smallskip
\item\label{H_{2}'} $\displaystyle \lim_{|s|\rightarrow
\infty}{h'(s)}/{|s|^{q-2}}=0$ for some $q\in (2,2^{*}_\alpha)$. \smallskip
\end{enumerate}
Then we have

\begin{theorem}\label{Morse}
Suppose that  \eqref{H_{0}}-\eqref{H_{1}'}-\eqref{H_{2}'}-\eqref{H_{3}}-\eqref{H_{4}} hold. Then there exists
$\lambda^{*}>0$ such that for $\lambda>\lambda^{*}$, the equation
(\ref{principal}) has at least $2\mathcal P_{1}(\Omega_{\lambda})-1$
solutions, if counted with their multiplicity.
\end{theorem}
Here $\mathcal P_{1}(\Omega_{\lambda})$ denotes the Poincar\'e polynomial of $\Omega_{\lambda}$
evaluated in $t=1$.
This  definition will be recalled later during the proof.
\medskip

To prove our results we use variational methods.
 Indeed a functional on a Hilbert space can be defined
in such a way that its critical points are exactly the solutions of \eqref{principal}. In this framework 
the assumption on $h$ are quite natural in order to deal with Nehari manifolds, Mountain Pass arguments
and Palais-Smale condition. We recall that if  $I$ is a $C^{1}$ functional on a Hilbert manifold $\mathcal M$
and $c\in \mathbb R$,
a sequence $\{v_{n}\}\subset \mathcal M$ is said to be a {\sl Palais-Smale sequence} for $I$ at level $c$
(briefly, a $(PS)_{c}$ sequence) if $I(v_{n})\to c$ and $I'(u_{n})\to 0$ in the  tangent bundle.
Furthermore, $I$ is said to satisfy the {\sl Palais-Smale condition} at level $c$
 if every $(PS)_{c}$ sequence has a convergent subsequence.

The functional related to our problem will turn out to be bounded from below 
on the ``manifold solution''
and verify the Palais-Smale
 condition at every level $c$, so 
the ``photography method'' of Benci and Cerami can be implemented
and the classical Ljusternick-Schnirelmann and Morse theory can be used to estimate the number of critical points
of the functional, that is, the number of solutions of \eqref{principal}.

 In the proof of our  results, we use some arguments that can be
found in \cite{Alves1,AFF,CGU}. However due to the presence of the Fractional Laplacian,
 some estimates more refined are need, such as in Lemma
 \ref{maximumprinciple}, Propositions \ref{4.1} and \ref{4.2}, for instance.

\medskip

\subsection{Notations.} 
Let us introduce here  few notations that will be used throughout  the paper.
\begin{itemize}
\item $B_{R}(x)$ denotes the open ball in $\mathbb R^{N}$
of radius $R$ centered in $x$; if $x=0$ we write $B_{R}.$
\item For $U\subset \mathbb R^{N},$ we denote with $\mathcal C_{U}$ the half
cylinder $U\times(0,+\infty)\subset \mathbb R^{N+1}.$ 
In particular $\mathcal C_{\mathbb R^{N}}=\mathbb R^{N}\times(0,+\infty).$
Whenever an element of $\mathcal C_{U}$ is written as $(x,y),$ it has always to be intended as $x\in U, y\in (0,+\infty).$
\item The lateral boundary of the cylinder is $\partial_{L} \mathcal C_{U}=\partial U \times[0,+\infty)$.

\end{itemize}
Other notations will be introduced along the paper as soon as we need. 
Finally, we will use $C_{1}, C_{2},\ldots$
to denote suitable positive constants, whose exact value may change from line to line.

\medskip

The plan of the paper is the following. In Section 
\ref{prelim} we recall some facts on the fractional Laplacian and 
write the variational framework in which we will work.
Section \ref{CC} is devoted to study the limit problem associated to our equation; in particular
 compactness results are proved  and, {\sl en passant}, also the existence of a ground state solution for \eqref{principal}.
 In Section \ref{Bary} we introduce the barycenter map and its properties. Moreover a careful
 analysis of the ground states level in terms of $\lambda$ is carried out. Finally, in Section \ref{1}
 we give the proof of Theorem \ref{Main2} and Theorem \ref{+1}, and finally 
 in Section \ref{2}, after recalling some facts and introducing some notations in classical Morse Theory,
 we prove Theorem \ref{Morse}.


\section{Preliminary results and the variational framework}\label{prelim}

\hspace{.45cm} In this section we start by introducing the functional framework necessary to apply variational methods
and recover some known results about the different forms of definition of the fractional power of the laplacian with Dirichlet boundary condition.

Let us consider the half cylinder with base $\Omega_{\lambda}$, i.e. $\mathcal C_{\Omega_{\lambda}}$ and let 
$$ H^1_{0,L}(\mathcal{C}_{\Omega_{\lambda}},y^{1-2\alpha}) 
= \Big\{v \in H^1(\mathcal{C}_{\Omega_{\lambda}}); \, v = 0 \, \, \mbox{on $\partial_L\mathcal{C}$ and $\|v\|_\alpha < \infty$}\Big\},$$
where
$$\|v\|_\alpha = \left( k_{\alpha}^{-1}\displaystyle\int_{\mathcal{C}_{\Omega_{\lambda}}} y^{1-2\alpha}|\nabla v|^2dxdy + \displaystyle\int_{\Omega_{\lambda}}|tr_{\Omega_{\lambda}}v(x)|^{2} dx \right)^\frac{1}{2},$$
 $k_{\alpha} = 2^{1-2\alpha}\Gamma(1-\alpha)/\Gamma(\alpha)$, $\alpha\in(0,1)$ and $tr_{\Omega_{\lambda}}$ is the trace operator given by  
$tr_{\Omega_{\lambda}} v = v(\cdot,0)$ for $v \in H^1_{0,L}(\mathcal{C}_{\Omega_{\lambda}},y^{1-2\alpha})$. It is not difficult to see that $H^1_{0,L}(\mathcal{C}_{\Omega_{\lambda}},y^{1-2\alpha})$ is a Hilbert space when endowed with the norm $\|\cdot\|_\alpha$, which comes from the following inner product
$$\langle v,w \rangle_\alpha = \int_{\mathcal{C}_{\Omega_{\lambda}}}k_\alpha^{-1}y^{1-2\alpha}\nabla v \nabla w dxdy + \int_{\Omega_{\lambda}}v(x,0)w(x,0) dx.$$

Consider the following space
$$\mathcal{V}_0^\alpha(\Omega_{\lambda}) = \Big\{ tr_{\Omega_{\lambda}} v; \, v \in H^1_{0,L}(\mathcal{C}_{\Omega_{\lambda}},y^{1-2\alpha})\Big\}.$$
By \cite[Proposition 2.1]{Sire}, there exists a trace operator from $H^1_{0,L}(\mathcal{C}_{\Omega_{\lambda}},y^{1-2\alpha})$ into the fractional Sobolev space $H^\alpha_0(\Omega_{\lambda})$.
Then  $\mathcal{V}_0^\alpha(\Omega_{\lambda})$ is a subspace of the fractional Sobolev space $H^{\alpha}(\Omega_{\lambda})$ and we consider it 
with the norm
$$
\|u\|_{\mathcal{V}_{0}^{\alpha}(\Omega_{\lambda})} = \left(\|u\|_{L^2(\Omega_{\lambda})}^2 + \int_{\Omega_{\lambda}} \int_{\Omega_{\lambda}}\frac{|u(x)-u(y)|^2}{|x-y|^{N+2\alpha}} dx dy\right)^{1/2}.
$$
Moreover, by Trace Theorem and embeddings of fractional Sobolev spaces (see  \cite[Theorem 6.7]{Nezza} for instance) it follows that
\begin{equation*}
\|tr_{\Omega_{\lambda}} v\|_{L^p(\Omega_{\lambda})} \leq C \|v\|_\alpha, \quad \forall v \in H^1_{0,L}(\mathcal{C}_{\Omega_{\lambda}},y^{1-2\alpha}),
\label{tracetheorem}
\end{equation*}
where $p \in (1,2^{*}_{\alpha}).$

By \cite[Proposition 2.1]{Sire} it holds that
$$\mathcal{V}_0^\alpha(\Omega_{\lambda}) = \left\{u \in L^2(\Omega_{\lambda}); \, u= \sum_{k = 1}^\infty b_k\varphi_k \, \, \mbox{ such that} \, \, \sum_{k = 1}^\infty b_k^2\mu_k^\alpha < \infty \right\},$$
where hereafter $(\mu_{k} ,\varphi_k)$ are the eigenpairs of $(-\Delta, H^1_0(\Omega_{\lambda}))$, 
$\mu_k$ repeated as much as its multiplicity.

Given $u \in C^\infty_0(\Omega_\lambda)$, with $u = \sum_{k = 1}^\infty b_k\varphi_k$, 
we define the operator
\begin{equation}\label{spectraldef}
(-\Delta)^\alpha u =  \sum_{k = 1}^\infty \mu_k^\alpha b_k \varphi_k
\end{equation}
which extends by density on
$\mathcal{V}_0^\alpha(\Omega_{\lambda})$.

Instead of working with this definition, 
we can get a local realization of $(-\Delta)^\alpha$
by adding one more dimension. Indeed, 
as proved in \cite[Section 2.1]{Sire}, for each $u \in \mathcal{V}_0^\alpha(\Omega_\lambda)$ there exists a unique 
$\tilde u \in H^1_{0,L}(\mathcal{C}_{\Omega_{\lambda}},y^{1-2\alpha})$, called the {\it $\alpha-$harmonic extension of $u$} such that
$$
\left\{
\begin{array}{ll}
-\mbox{div}(y^{1-2\alpha}\nabla \tilde u) = 0 & \mbox{in } \mathcal{C}_{\Omega_{\lambda}}\\
\tilde u = 0 & \mbox{on $\partial_L \mathcal{C}_{\Omega_{\lambda}}$}\\
\tilde u(\cdot,0) = u & \mbox{on $\Omega_{\lambda}$}.
\end{array}\right.
$$
Moreover, if $u = \sum_{k=1}^\infty b_k\varphi_k$ then
\begin{equation}
\tilde u(x,y) = \sum_{k=1}^\infty b_k\varphi_k(x)\psi(\mu_k^{1/ 2}y), \quad \forall (x,y) \in \mathcal{C}_{\Omega_{\lambda}},
\label{harmonicextension}
\end{equation}
where $\psi$ solves the Bessel equation
\begin{equation}\label{besselequation}
\left\{
\begin{array}{ll}
\psi''(s)+\displaystyle\frac{1-2\alpha}{s}\psi'(s) = \psi\,, & s>0\medskip \\ 
-\displaystyle\lim_{s\rightarrow0^+}s^{1-2\alpha}\psi{'}(s) = k_{\alpha} & \medskip\\
\psi(0) = 1. &
\end{array}\right.
\end{equation}
Now, for a fixed  $u \in \mathcal{V}_0^\alpha(\Omega_\lambda)$ 
 define the functional 
 $\displaystyle \left.\frac{1}{k_{\alpha}}\frac{\partial \tilde u}{\partial y^{\alpha}}\right|_{\Omega_\lambda \times \{0\}} 
 \in \mathcal{V}^{\alpha}_ 0(\Omega_\lambda)^*$ by
$$
\left< \frac{1}{k_{\alpha}}\frac{\partial \tilde u}{\partial y^{\alpha}}(\cdot,0), g \right>_{(\mathcal{V}^{\alpha}_ 0(\Omega_\lambda)^*,\mathcal{V}^{\alpha}_ 0(\Omega_\lambda))} := \frac{1}{k_\alpha}\int_{\mathcal{C}_{\Omega_{\lambda}}} y^{1-2\alpha}\nabla \tilde u \nabla \tilde g \,dxdy, \ \ \
g \in \mathcal{V}_0^\alpha(\Omega_{\lambda}).
$$
Integration by parts in the right hand side of the last equality explains the notation chosen to the functional, since
$$\left< \frac{1}{k_{\alpha}}\frac{\partial \tilde u}{\partial y^{\alpha}}(\cdot,0), g \right>_{(\mathcal{V}^{\alpha}_ 0(\Omega_\lambda)^*,\mathcal{V}^{\alpha}_ 0(\Omega_\lambda))} = \left< \frac{1}{k_{\alpha}}\frac{\partial \tilde u}{\partial y^{\alpha}}(\cdot,0), g \right>_{L^2(\Omega_\lambda)}
$$
for all $g \in \mathcal{V}^{\alpha}_ 0(\Omega_\lambda)$,
where
$$\frac{\partial \tilde u}{\partial y^{\alpha}}(x,0) =
 -\displaystyle\lim_{y\rightarrow0^+}y^{1-2\alpha}\frac{\partial \tilde u}{\partial y}(x,y)\
 \quad \forall x\in\Omega_{\lambda}.
$$

Then we can define an operator $A_\alpha:\mathcal{V}_0^\alpha(\Omega_{\lambda}) \to \mathcal{V}_0^\alpha(\Omega_{\lambda})^*$ such that
$$A_\alpha u := 
 \left.\frac{1}{k_{\alpha}}\frac{\partial \tilde u}{\partial y^{\alpha}}\right|_{\Omega_{\lambda} \times \{0\}}.$$
 Let us prove that the operators $A_\alpha$ and $(-\Delta)^\alpha$ 
 defined in \eqref{spectraldef} are in fact the same, i.e., that for all $u \in \mathcal{V}_0^\alpha(\Omega_{\lambda})$,
$$A_\alpha u = \sum_{k = 1}^\infty \mu_k^\alpha b_k\varphi_k, \quad \text{ where } \quad  u = \sum_{k = 1}^\infty b_k\varphi_k.$$
It is enough to show that for all $u \in \mathcal{V}_0^\alpha(\Omega_{\lambda})$,
$$\left< \frac{1}{k_{\alpha}}\frac{\partial \tilde u}{\partial y^{\alpha}}(\cdot,0), \varphi_k \right>_{(\mathcal{V}^{\alpha}_ 0(\Omega_\lambda)^*,\mathcal{V}^{\alpha}_ 0(\Omega_\lambda))} = \left< (-\Delta)^\alpha u,\varphi_k\right>_{L^2(\Omega_\lambda)}, \quad \mbox{for all $k \in \mathbb{N}$.}$$
For $u \in \mathcal{V}_0^\alpha(\Omega_{\lambda})$ and $k \in \mathbb{N}$, 
by (\ref{harmonicextension}), 
$$\tilde u(x,y) = \sum_{k=1}^\infty b_k\varphi_k(x)\psi(\mu_k^{1/ 2}y)\quad \text{ and  }\quad
\widetilde{\varphi_k} (x,y) = \varphi_k(x)\psi(\mu_k^{1/ 2}y).$$ 
Now, integration by  parts 
implies that, for  $y > 0$,
$$\int_{\Omega_\lambda}y^{1-2\alpha}\nabla \tilde u(x,y) \nabla \widetilde{\varphi_k}(x,y) dx 
= y^{1-2\alpha}b_k\left(\mu_k\psi(\mu_k^{1/2}y)^2 + \psi'_k(\mu_k^{1/2}y)^2\right).
$$
Then, by (\ref{besselequation})
\begin{eqnarray*}
\left< \frac{1}{k_{\alpha}}\frac{\partial \tilde u}{\partial y^{\alpha}}(\cdot,0), \varphi_k \right>_{(\mathcal{V}^{\alpha}_ 0(\Omega_\lambda)^*,\mathcal{V}^{\alpha}_ 0(\Omega_\lambda))} & = & 
\frac{1}{k_\alpha}\int_{\mathcal{C}_{\Omega_{\lambda}}} y^{1-2\alpha}\nabla \tilde u\nabla \widetilde{\varphi_k}\, dxdy\\
& = &\frac{1}{k_{\alpha}} \int_0^{+\infty}y^{1-2\alpha}b_k\left(\mu_k\psi(\mu_k^{1/2}y)^2 + \psi'_k(\mu_k^{1/2} y)^2\right)dy\\
& = &\frac{1}{k_{\alpha}} \left.\lim_{\eta \to 0^+}y^{1-2\alpha}\mu_k^{1/2}b_k\psi'(\mu_k^{1/2}y)\psi(\mu_k^{1/2}y)\right|_{y = \eta}\\
& = & b_k\mu_k^\alpha\\
& = & \left< (-\Delta)^\alpha u,\varphi_k\right>_{L^2(\Omega_\lambda)}.
\end{eqnarray*}
Hence,  in (\ref{principal}) we are going to understand $(-\Delta)^\alpha$ as $A_\alpha$. 

\medskip

Let us pass to the definition of weak solution for problems involving the fractional Laplacian.
We say that a function $u$ 
is  a solution of the linear problem 
$$
\left\{
\begin{array}{ll}
(-\Delta)^{\alpha} u = f(x)  & \mbox{in $\Omega_\lambda$}\\
u = 0 & \mbox{on $\partial\Omega_\lambda$},
\end{array} \right.
$$
where $f \in \mathcal{V}_0^\alpha(\Omega_{\lambda})^*$, if $u = tr_{\Omega_\lambda} v$, where $v \in H^1_{0,L}(\mathcal{C}_{\Omega_{\lambda}},y^{1-2\alpha})$ 
is a solution of
\begin{equation*}
\left\{
\begin{array}{ll}
-\mbox{div}(y^{1-2\alpha}\nabla v) = 0 & \mbox{ in } \mathcal{C}_{\Omega_{\lambda}}\\
v = 0 & \mbox{ on }\partial_L \mathcal{C}_{\Omega_{\lambda}}\\
\displaystyle \frac{1}{k_\alpha}\frac{\partial v}{\partial y^\alpha}(x,0) = f(x) &\  x \in \Omega_{\lambda}.
\end{array} \right.
\label{plinear}
\end{equation*}

Analogously, we say that  $u \in \mathcal{V}^{\alpha}_ 0(\Omega_{\lambda})$ is a weak solution of (\ref{principal}) if $u = tr_{\Omega_{\lambda}} v$, where $v \in H^1_{0,L}(\mathcal{C}_{\Omega_{\lambda}},y^{1-2\alpha})$ is a weak
solution of
\begin{equation*}
\left\{
\begin{array}{ll}
-\mbox{div}(y^{1-2\alpha}\nabla v) = 0 & \mbox{ in } \mathcal{C}_{\Omega_{\lambda}}\\
v = 0 & \mbox{ on }\partial_L \mathcal{C}_{\Omega_{\lambda}}\\
\displaystyle \frac{1}{k_\alpha}\frac{\partial v}{\partial y^\alpha}+v(x,0) = h(v(x,0)) &\  v>0 \ \text{ in }\ \Omega_{\lambda},
\end{array} \right.
\label{plinear}
\end{equation*}
that is,
\begin{equation*}\label{weakform}
\int_{\mathcal{C}_{\Omega_{\lambda}}}k_\alpha^{-1} y^{1-2\alpha}\nabla v \nabla \psi dxdy + \int_{\Omega_{\lambda}} v(x,0) \psi(x,0) dx= \int_{\Omega_{\lambda} } h(v(x,0))\psi(x,0)dx, \quad \forall \psi \in H^1_{0,L}(\mathcal{C}_{\Omega_{\lambda}},y^{1-2\alpha}).
\end{equation*}
As it is easy to see, this is equivalent to say that $v$ is  a critical point of the $C^{1}$ functional
$$
I_{\lambda}(v)=\frac{k_\alpha^{-1}}{2}\int_{\mathcal{C}_{\lambda}} y^{1-2\alpha}|\nabla v|^{2}dxdy + 
\frac{1}{2}\int_{\Omega_{\lambda}} |v(x,0)|^{2} dx - \displaystyle\int_{\Omega_{\lambda} } H(v(x,0))dx
$$
in $H^1_{0,L}(\mathcal{C}_{\Omega_\lambda},y^{1-2\alpha}).$


\medskip

It is not difficult to see that, in virtue of the assumptions on the nonlinearity
$h$, the functional $I_{\lambda}$ possesses a Mountain Pass Geometry:
the mountain pass level will be denoted with $c( \Omega_{\lambda})>0.$
We also define the
{\sl Nehari manifold} associated to $I_{\lambda}$  by
\begin{equation}\label{N}
{\cal{M}}_{\lambda}= \Big\{v \in H^1_{0,L}(\mathcal{C}_{\Omega_{\lambda}},y^{1-2\alpha})\setminus \{0\}: J_{\lambda}(v)=0\Big\}
\end{equation}
where
$$
J_{\lambda}(v):=I'_{\lambda}(v)[v]=k_\alpha^{-1}
\int_{\mathcal{C}_{\Omega_\lambda}} y^{1-2\alpha}|\nabla v|^{2}dxdy + 
\int_{\Omega_{\lambda}} |v(x,0)|^{2} dx - \int_{\Omega_{\lambda} } h(v(x,0))v(x,0)dx.
$$
We will need the following  properties about $\mathcal M_{\lambda}.$ They are standard, as well, and just based
on the hypothesis made on the nonlinearity; for a proof
one can follow e.g. \cite[Lemma 2.2]{BC3}.
\begin{lemma}\label{Nehariman} Let $\lambda>0$. The following propositions hold true:
\begin{itemize}
\item[1.] for every $u\in \mathcal M_{\lambda}$ it is $J_{\lambda}'(u)[u]<0;$ \smallskip
\item[2.] $\mathcal M_{\lambda}$ is a differentiable manifold radially diffeormorphic to the unit sphere 
$\mathrm S$ in $H^{1}_{0,L}(\mathcal C_{\Omega_{\lambda}}, y^{1-2\alpha})$ and bounded away from $0$; \smallskip
\item[3.] $I_\lambda$ is bounded from below on $\mathcal M_{\lambda}$ and 
\begin{equation}\label{clambda}
0<c(\Omega_{\lambda})=\inf_{\mathcal M_{\lambda}} I_{\lambda}=\inf_{u\neq0} \sup_{t>0} I_{\lambda}(tu).
\end{equation}
\end{itemize}
\end{lemma}
In particular every nonzero function $v\in H^{1}_{0,L}(\mathcal C_{\Omega_{\lambda}},y^{1-2\alpha})$ 
can be ``projected''  on $\mathcal M_{\lambda}$;
in other words we have an homeomorphism which just multiply a function by a positive constant
(depending on the function)
\begin{equation}\label{mappanehari}
v\in  H^1_{0,L}(\mathcal{C}_{\Omega_{\lambda}},y^{1-2\alpha})\setminus \{0\}\longmapsto t_{\lambda} v \in 
\mathcal M_{\lambda}. 
\end{equation}
It is  clear that $\mathcal M_{\lambda}$ is a natural constraint for $I_{\lambda}$ in the sense that
\begin{corollary}\label{17}
If $v$ is a critical point of $I_{\lambda}$ on ${\cal{M}}_{\lambda}$, then $v$ is a nontrivial critical point of $I_{\lambda}$ on $H^1_{0,L}(\mathcal{C}_{\Omega_{\lambda}},y^{1-2\alpha})$.
\end{corollary}
Moreover, standard arguments show that the Palais-Smale sequences for $I_{\lambda}$ restricted to $M_{\lambda}$
are Palais-Smale sequences for the free functional $I_{\lambda}$, and  $I_{\lambda}$
satisfies the Palais-Smale condition on $\mathcal M_{\lambda}$ if and only if it satisfies the 
same condition on $H^{1}_{0,L}(\mathcal C_{\Omega_{\lambda}}, y^{1-2\alpha})$.

\begin{remark} \label{rem}
In the next sections we will use some auxiliary functionals: they differ from $I_{\lambda}$
just for the domain on which these functionals are defined.
In a similar way as in \eqref{N} we will define the  Nehari manifolds related to these functionals
and  it is clear 
that analogous properties to all those stated  on $\mathcal M_{\lambda}$ hold, 
since they are essentially  based on the structure
of the functional,  on the hypothesis made on the nonlinearity, and on how the Nehari manifold is defined.
For this reason, the above cited properties will be used without any other comment through the paper.
\end{remark}

\section{Compactness results and existence of a ground state solution for $I_{\lambda}$}\label{CC}

Now let us consider the half cylinder with base $\mathbb{R}^{N}$, 
$\mathcal{C}_{\mathbb R^{N}}$,
and define 
$$H^1(\mathcal C_{\mathbb R^{N}},y^{1-2\alpha}) = \{v \in H^1(\mathcal C_{\mathbb R^{N}}): \, \, \|v\|_{\mathcal C_{\mathbb R^{N}}} 
< \infty\},$$
where
$$\|v\|_{\mathcal C_{\mathbb R^{N}}} = \left( k_{\alpha}^{-1}\displaystyle\int_{\mathcal{C}_{\mathbb R^{N}}} y^{1-2\alpha}|\nabla v|^2dxdy + 
\displaystyle\int_{\mathbb{R}^{N}}|v(x,0)|^{2} dx \right)^{1/2}.$$
It is easy to see that $H^1(\mathcal C_{\mathbb R^{N}},y^{1-2\alpha})$ is a Hilbert space when endowed with the norm 
$\|\cdot\|_{\mathcal C_{\mathbb R^{N}}}$, which comes from the following inner product
$$\langle v,w \rangle_{\mathcal C_{\mathbb R^{N}}} = k_{\alpha}^{-1}\displaystyle\int_{\mathcal{C}_{\mathbb R^{N}}} y^{1-2\alpha}\nabla v \nabla wdxdy + 
\displaystyle\int_{\mathbb{R}^{N}}v(x,0)w(x,0) dx.$$
An important result we are going to use in this work is related with
the existence of a positive ground state solution of the {\sl limit problem}
\begin{equation}\label{limite}\tag{$P_{\infty}$}
(-\Delta)^{\alpha} u +u= h(u)  \ \ \mbox{in} \ \ \mathbb{R}^N,
\end{equation}
i.e., the least energy solution for the functional
$$
I_{\infty}(v)=\frac{k_\alpha^{-1}}{2}\int_{\mathcal{C}_{\mathbb R^{N}}}y^{1-2\alpha}|\nabla v|^{2} dxdy+ 
\displaystyle\int_{\mathbb{R}^N} |v(x,0)|^{2} dx - \int_{\mathbb{R}^N} H(v(x,0))dx.
$$
It is standard to see that $I_{\infty}$ has a Mountain Pass Geometry in $H^1(\mathcal C_{\mathbb R^{N}},y^{1-2\alpha})$, whose mountain pass level is denoted by $c(\mathbb R^{N})>0.$
Moreover, we can define the Nehari manifold associated to $I_{\infty}$ 
$$\mathcal M_{\infty}=\Big\{v\in H^{1}(\mathcal C_{\mathbb R^{N}}, y^{1-2\alpha}): I_{\infty}'(v)[v]=0\Big\}$$
and standard computations give
$$0<c(\mathbb R^{N})
=\inf_{\mathcal{M}_{\infty}}I_{\infty}.
$$
%
%
%
%
The theorem below states the existence of a ground
state solution for \eqref{limite}, hence $c(\mathbb R^{N})$ is achieved
on a function of mountain pass type.
The result is  known in the literature (it can be obtained with similar arguments used in \cite[Theorem 3.1]{Alves1})
but for completeness, and since it will be very useful for us, we prefer to give the proof.


\begin{lemma}\label{gio16}
Let $\{v_{n}\} \subset {\cal{M}}_{\infty}$ be a sequence  satisfying
$I_{\infty}(v_{n})\rightarrow c(\mathbb R^{N})$. Then, eighter 
\begin{itemize}
\item[a)] $\{v_{n}\}$ has a strongly convergent subsequence in
$H^1(\mathcal C_{\mathbb R^{N}},y^{1-2\alpha})$ 
\end{itemize}
 or 
 \begin{itemize}
\item[b)] there exists a sequence $\{x_{n}\} \subset \mathbb R^{N}$
 such that, up to a subsequence, $|x_{n}|\rightarrow +\infty$ and $\overline{v}_{n}(x,y):=v_{n}(x-x_{n},y)$ strongly converges in $H^1(\mathcal C_{\mathbb R^{N}},y^{1-2\alpha})$.
\end{itemize}
In particular, there exists a positive minimizer, hereafter denoted by $\mathfrak w_{\infty}$, for $c(\mathbb R^{N})$.
\end{lemma}
\begin{proof} By the Ekeland Variational Principle we can assume without loss of generality that $\{v_n\}$ is a  $(PS)_{c(\mathbb R^N)}$ 
sequence for $I_\infty$ on $\mathcal{M}_\infty$ and then, by very known arguments, it follows that it is a $(PS)_{c(\mathbb R^N)}$ 
sequence to $I_\infty$ on $H^1(\mathcal C_{\mathbb R^{N}},y^{1-2\alpha})$ . In a standard way one can prove that $\{v_{n}\}$ is 
bounded in $H^1(\mathcal C_{\mathbb R^{N}},y^{1-2\alpha})$ and then, up to a subsequence, $v_{n}\rightharpoonup v$ 
in $H^1(\mathcal C_{\mathbb R^{N}},y^{1-2\alpha})$. 

\medskip

{\sl First case:} $v \neq 0$. It is a simple matter to prove in this case  $I'_\infty(v) = 0$.
It follows from the  Fatou Lemma, \eqref{H_{3}} and the weak lower semicontinuity of the norm that 
\begin{eqnarray*}
c(\mathbb R^N) & \leq & I_\infty(v)\\
& = & I_\infty(v) - \frac{1}{\theta}I'_\infty(v)[v]\\
& = & \left(\frac{1}{2} - \frac{1}{\theta}\right)\|v\|^2_{\mathcal C_{\mathbb R^{N}}} + \int_{\mathbb{R}^N}\left(\frac{1}{\theta}h(v)v - H(v)\right)dx\\
& \leq & \liminf_{n \to \infty}\left[ \left(\frac{1}{2} - \frac{1}{\theta}\right)\|v_n\|^2_{\mathcal C_{\mathbb R^{N}}} + \int_{\mathbb{R}^N}\left(\frac{1}{\theta}h(v_n)v_{n} - H(v_n)\right)dx\right]\\
& = & c(\mathbb R^N),
\end{eqnarray*}
which implies that $I_\infty(v) = c(\mathbb R^N)$. Now let us prove that $v_n \to v$ in $H^1(\mathcal C_{\mathbb R^{N}},y^{1-2\alpha})$
 and for this it is enough to show that  $\|v_n\|_{\mathcal C_{\mathbb R^{N}}} \to \|v\|_{\mathcal C_{\mathbb R^{N}}}$. 
 By the weak semicontinuity of the norm it follows that 
\begin{equation}
\|v\|_{\mathcal C_{\mathbb R^{N}}} \leq \liminf_{n \to \infty} \|v_n\|_{\mathcal C_{\mathbb R^{N}}}.
\label{liminf}
\end{equation}
Supposing by contradiction that 
$$\limsup_{n \to \infty} \|v_n\|_{\mathcal C_{\mathbb R^{N}}} > \|v\|_{\mathcal C_{\mathbb R^{N}}},$$
Fatou Lemma implies that 
\begin{eqnarray*}
c(\mathbb R^N) & = & \limsup_{n \to \infty}\left(\frac{1}{2} - \frac{1}{\theta}\right)\|v_n\|_{\mathcal C_{\mathbb R^{N}}}^2 + \limsup_{n \to \infty}\int_{\mathbb{R}^N}\left(\frac{1}{\theta}h(v_n)v_n - H(v_n)\right)dx\\
& > & \left(\frac{1}{2} - \frac{1}{\theta}\right)\|v\|_{\mathcal C_{\mathbb R^{N}}}^2 + \int_{\mathbb{R}^N}\left(\frac{1}{\theta}h(v)v - H(v)\right)dx\\
& = & c(\mathbb R^N),
\end{eqnarray*}
which is a contradiction. Then it follows that 
$$\limsup_{n \to \infty} \|v_n\|_{\mathcal C_{\mathbb R^{N}}}\leq\|v\|_{\mathcal C_{\mathbb R^{N}}}$$
and this together with (\ref{liminf}) implies that $v_n \to v$ in $H^1(\mathcal C_{\mathbb R^{N}},y^{1-2\alpha})$.

\medskip

{\sl Second case: $v = 0$.}
Then $\{v_n\}$ is not strongly convergent; indeed, if this were not the case,
we would have a contradiction  with the fact that $I_\infty(v_n) \to c(\mathbb R^N) > 0$. Hence there are $R,\gamma>0$ and 
$\{x_{n}\}\subset\mathbb{R}^{N}$ such that, up to a subsequence 
$$
\int_{B_{R}(x_{n})}|v_{n}(x,0)|^{2} dx \geq \gamma >0
$$
In fact, on the contrary, by the version of concentration compactness principle given in \cite[Lemma 2.2]{Felmer}, $v_n(\cdot,0) \to 0$ in $L^q(\mathbb{R}^N)$ for $2 < q < 2^*_\alpha$. By this fact together with conditions  \eqref{H_{0}}-\eqref{H_{4}}, implies that
$$
I_\infty(v_n) = \int_{\mathbb{R}^N}\left(\frac{1}{2}h(v_n(x,0))v_n(x,0) - H(v_n(x,0))\right)dx + o_n(1) = o_n(1),
$$
which contradicts again $I_{\infty}(v_{n})\rightarrow c(\mathbb R^{N})> 0$. Moreover, since $v = 0$, it follows that $|x_{n}|\rightarrow +\infty$. This follows because otherwise Sobolev embedding can be used to prove that $v \neq 0$.
Since $\mathbb{R}^{N}$ is invariant by translation, defining $\overline{v}_{n}(x,y):=v_{n}(x-x_{n},y)$ we still have
 a $(PS)_{c(\mathbb R^N)}$ 
 sequence for $I_\infty$, which is contained on $\mathcal{M}_\infty$
 and is bounded in $H^1(\mathcal{C}_{\mathbb R^{N}},y^{1-2\alpha})$. Then $\overline{v}_{n}\rightharpoonup \overline{v}\neq 0$ and hence, by the first case, $\overline{v}_n \to \overline{v}$ 
 in $H^1(\mathcal{C}_{\mathbb R^{N}},y^{1-2\alpha})$, $I_\infty(\overline{v}) = c(\mathbb R^N)$ and $\overline{v}$ is a ground state for $I_\infty$.
\end{proof}

For what concerns our functional we have

\begin{lemma} For every $\lambda>0$, the functional $I_{\lambda}$ satisfies the Palais-Smale condition 
on $H^1_{0,L}(\mathcal{C}_{\Omega_{\lambda}},y^{1-2\alpha})$, and hence on $\mathcal M_{\lambda}$.
 \label{lema_PS_unconstrained}
\end{lemma}
\begin{proof}
Let $\{v_{n}\}\subset H^1_{0,L}(\mathcal{C}_{\Omega_{\lambda}},y^{1-2\alpha})$ be a sequence such that
$$
I_{\lambda}(v_{n})\rightarrow c\quad \mbox{and} \quad
I'_{\lambda}(v_{n})\rightarrow 0.
$$
Thus, by \eqref{H_{3}} we get 
$$
C_{1}+o_{n}(1)\|v_{n}\|_{\alpha} \geq  I_{\lambda}(v_{n})-\frac{1}{\theta} I'_{\lambda}(v_{n}) [v_{n}] \geq   
\left(\frac{1}{2}-\frac{1}{\theta}\right)\|v_{n}\|^{2}_{\alpha},
$$
which gives that $\{v_{n}\}$ is bounded in
$ H^1_{0,L}(\mathcal{C}_{\Omega_{\lambda}},y^{1-2\alpha})$. 
Then we may assume that, up to a subsequence,
 $v_{n}\rightharpoonup v$ in
$ H^1_{0,L}(\mathcal{C}_{\Omega_{\lambda}},y^{1-2\alpha})$ and 
hence $tr_{\Omega_\lambda}v_{n}\rightarrow tr_{\Omega_\lambda}v$ in
$L^{s}(\Omega_{\lambda})$, with $2\leq s < 2^{*}_{\alpha}$. Thus, since the nonlinearity $h$ has subcritical growth, by standard calculations, 
we see that $I_{\lambda}$ satisfies the Palais-Smale condition.
\end{proof}

Then, taking into account that $I_{\lambda}$
is  bounded from below on $\mathcal M_{\lambda}$ we have
\begin{theorem}\label{existence}
For every $\lambda>0$, 
$c(\Omega_{\lambda})$ is achieved on a ground state solution denoted with  $\mathfrak w_{\Omega_{\lambda}}.$
\end{theorem}
%

\section{The Barycenter map and behavior of the mountain pass levels}\label{Bary}

In this section, we study the behavior of some minimax levels with respect to the parameter $\lambda$. 
To do so, some preliminaries are in order.

Without any loss of generality, from now on we assume that $0\in \Omega_{\lambda}$.
Following  \cite{BC3}, for $v
\in H^1_{0,L}(\mathcal{C}_{\Omega_{\lambda}},y^{1-2\alpha})$ with compact support
and such that $tr_{\Omega_\lambda}v^{+} \not\equiv 0$, we  define the {\sl barycenter} or {\sl center of mass}
 of $v$ in the following way: first consider the ``trivial'' extension of $v^{+}(\cdot,0)=tr_{\Omega_\lambda}v^{+}$
 to the whole $\mathbb R^{N}$ (denoted by the same symbol) and then set 
$$
\beta(v):=\beta (v^{+}(\cdot,0))=\frac{\displaystyle\int_{\mathbb{R}^N}x |v^{+}(x,0)|^{2}dx}{\displaystyle\int_{\mathbb{R}^N}|v^{+}(x,0))|^{2}dx}
\in \mathbb R^{N}.
$$
For  $R>r>0$ let
us denote by $A_{R,r}(\tilde x)$ the open anulus in $\mathbb R^{N}$ centered in $\tilde x$
$$
A_{R,r}(\tilde x)=B_{R}(\tilde x)\setminus \overline{B}_{r}(\tilde x).
$$
Define the functional on $H^{1}_{0,L}(\mathcal C_{A_{\lambda R, \lambda r}(\tilde x)}, y^{1-\alpha})$
\begin{multline}\label{chap1}
\widehat{I}_{\lambda, \tilde x}(v)= \frac{1}{2}\int_{\mathcal{C}_{A_{\lambda
R,\lambda r}(\tilde x)}} y^{1-2\alpha}|\nabla v|^{2} dxdy + \frac{1}{2} \int_{A_{\lambda R,\lambda r}(\tilde x)}|v(x,0)|^{2} dx \\ 
-   \int_{A_{\lambda R,\lambda r}(\tilde x)} H(v(x,0)) dx,
\end{multline}
 and set
 \begin{eqnarray}
\widehat{\mathcal{M}}_{\lambda, \tilde x}&=&\Big\{v \in 
H^1_{0,L}(\mathcal{C}_{A_{\lambda
R,\lambda r}(\tilde x)},y^{1-2\alpha})\setminus
\{0\}; \, \widehat{I}'_{\lambda,\tilde x}(v)[v]=0 \Big\} \label{chap2}\\
a(R,r,\lambda,\tilde x)&=&\inf\Big\{\widehat{I}_{\lambda,\tilde x}(v):
v\in\widehat{{\cal{M}}}_{\lambda,\tilde x}\;\;\mbox{and}\;\;
\beta(v)=\tilde x\Big\}. \label{chap3}
 \end{eqnarray}
As is customary,  when $\tilde x=0$ we simply write $\widehat{I}_{\lambda}$, $\widehat{\mathcal{M}}_{\lambda }$ and
$a(R,r,\lambda)$.
We observe that the value $a(R,r,\lambda,\tilde x)$ does not depend on the ``center'' $\tilde x$.

Since $\widehat I_{\lambda,\tilde x}$ has the Mountain Pass Geometry, is bounded from
below on $\widehat{{\cal{M}}}_{\lambda,\tilde x}$ and satisfies the Palais-Smale condition,
the infima $a(R,r,\lambda,\tilde x)$ are obtained.

In the following we use a version of a maximum principle to the operator $(-\Delta)^\alpha$. 
Since we were not able to find in the literature the exact version of it which is necessary here, we prove it in the following result.

\begin{lemma}\label{maximumprinciple}
Let $\Gamma \subset \mathbb{R}^N$ be a smooth domain 
and  $v \in H^1_{0,L}(\mathcal{C}_{\Gamma},y^{1-2\alpha})$   such that 
\begin{equation}
\left\{
\begin{array}{ll}
-\mbox{div} (y^{1-2\alpha}\nabla v) = 0 & \mbox{in $\mathcal{C}_{\Gamma}$}\\
v = 0 & \mbox{on $\partial_L \mathcal{C}_{\Gamma}$}\\
\displaystyle \frac{1}{k_\alpha}\frac{\partial v}{\partial y^\alpha}(x,0) + v(x,0) = f(x) & \mbox{on $\Gamma$}.
\end{array}\right.
\label{Plinear}
\end{equation}
in the weak sense. 
If $f \geq 0$, then $v \geq 0$ in $\mathcal{C}_{\Gamma}$.
\end{lemma}
\begin{proof}
Since $v$ satisfies (\ref{Plinear}), it follows that for all $\psi \in H^1_{0,L}(\mathcal{C}_{\Gamma},y^{1-2\alpha})$ such that 
$\psi \geq 0$ in $\partial_L \mathcal{C}_{\Gamma}$, we have
$$
k_\alpha^{-1}\int_{\mathcal{C}_{\Gamma}} y^{1-2\alpha}\nabla v \nabla \psi dxdy + \int_{\Gamma} v(x,0) \psi(x,0) dx= \int_{\Gamma} f(x)\psi(x,0)dx.
$$
If we take $v^-$ (where $v = v^+ + v^-$) as a test function in the last expression we get
$$
k_\alpha^{-1}\int_{\mathcal{C}_{\Gamma}} y^{1-2\alpha}|\nabla v^-|^2 dxdy + \int_{\Gamma} |v^-(x,0)|^2 dx = \int_{\Gamma} f(x)v^-dx \leq 0.
$$
But this implies that $v^- \equiv 0$ and then $v \geq 0$.
\end{proof}

The next result will be useful in future estimates with the barycenter map.
\begin{proposition}\label{4.1}
The number $a(R,r,\lambda)$ satisfies
$$
\displaystyle\liminf_{\lambda\rightarrow \infty}a(R,r,\lambda) >
c(\mathbb R^N).
$$
\end{proposition}
\begin{proof} From the definition of $a(R,r,\lambda)$ and $c(\mathbb R^N)$, we get
$$
a(R,r,\lambda) > c(\mathbb R^N).
$$
Suppose by contradiction that there exist $\lambda_{n}\rightarrow \infty$ 
such that $a(R,r,\lambda_{n})\rightarrow c(\mathbb R^N)$.
Since $a(R,r,\lambda_{n})$ is reached there exist
$v_{n}\in \widehat{\mathcal{M}}_{\lambda_{n}}$ such that 
$$
\beta(v_{n})=0 \quad \mbox{and}  \quad
\widehat{I}_{\lambda}(v_n) = a(R,r,\lambda_{n})\rightarrow c(\mathbb R^N).
$$

Since $h \geq 0$, by \eqref{H_{0}} and  Lemma \ref{maximumprinciple} it is  $v_n \geq 0$
for all $n \in \mathbb{N}$. Moreover, since $v_n=0$ on 
 $\partial_{L} \mathcal C_{A_{\lambda_n R, \lambda_n r}}$, 
by considering the trivial extension 
on $\mathcal C_{\mathbb R^{N}}\setminus \mathcal C_{A_{\lambda_n R, \lambda_n r}}$
(which we denote with the same symbol) we obtain a function in 
$H^{1}_{0,L}(\mathcal C_{\mathbb R^{N}}, y^{1-2\alpha})$. Consequently,
$$
v_n \rightharpoonup 0 \,\,\, \mbox{in} \,\,\, H^1(\mathcal{C}_{\mathbb R^{N}},y^{1-2\alpha}), \,\, I_{\infty}(v_n)=
a(R,r,\lambda_{n}) \rightarrow c(\mathbb R^N) \,\,\, \mbox{and} \,\,\, v_n \in \mathcal{M}_{\infty}.
$$
Recalling that $c(\mathbb R^N)>0$, we have that $\{v_{n}\}$ is not strongly convergent. From Lemma  \ref{gio16}, we get
(recall $z=(x,y)$)
$$
v_{n}(z)=w_{n}(z+z_n)+\mathfrak w_{\infty}(z+z_{n})
$$
where $\{w_{n}\} \subset H^1(\mathcal{C}_{\mathbb R^{N}},y^{1-2\alpha})$ is a sequence converging strongly to $0$,
$\{z_{n}\}=\{(x_{n},0)\}\subset \mathbb{R}^{N+1}$ is such that $|x_{n}|\rightarrow \infty$ and 
$\mathfrak w_{\infty}\in H^1(\mathcal {C}_{\mathbb R^{N}},y^{1-2\alpha})$ is a positive function verifying
$$
I_{\infty}(\mathfrak w_{\infty})=c(\mathbb R^N) \quad \mbox{and} \quad
I'_{\infty}(\mathfrak w_{\infty})=0.
$$
Since $I_{\infty}$ is rotationally invariant on functions of type $w(\cdot, 0),$
 we can assume that
$$
z_{n}=(x^{1}_{n}, 0,0,\ldots,0)\ \ \text{and } \ \ x^{1}_{n}<0.
$$
Now we set
$$
M=\int_{\mathbb{R}^N}|\mathfrak w_{\infty}(x,0)|^{2} dx>0.
$$
Since $\|w_{n}\|_{\alpha}\rightarrow 0$, it follows that 
$$
\int_{B_{r\lambda_{n}/2}(x_{n})}|
w_{n}(x+x_{n},0)+\mathfrak w_{\infty}(x + x_{n},0)|^{2} dx\rightarrow M,
$$
from which we obtain
$$
\int_{\Theta_{n}}| v_{n}(x,0)|^{2} dx\rightarrow M, \ \ \text{ where }\ \ \Theta_{n}=B_{r\lambda_{n}/2}(x_{n})\cap A_{\lambda_{n} R, \lambda_{n} r} 
$$
 and hence
\begin{eqnarray}\label{1contra}
\int_{\Upsilon_{n}}|v_{n}(x,0)|^{2}dx\rightarrow 0, \ \ \text{ where }\ \ \Upsilon_{n}= A_{\lambda_{n} R, \lambda_{n} r
 \setminus B_{\lambda_{n}r/2}(x_{n}).}
\end{eqnarray}
Since
$\beta(v_{n})=0$, we get
$$
0=\int_{A_{\lambda_{n}R,\lambda_{n}r}}x^{1}|
v_{n}(x,0)|^2 dx=\int_{\Theta_{n}}x^{1}|
v_{n}(x,0)|^2 dx+\int_{\Upsilon_{n}}x^{1}| v_{n}(x,0)|^2 dx.
$$
Thus,
$$
-\frac{r\lambda_{n}}{2}(M+o_{n}(1))+R\lambda_{n}\int_{\Upsilon_{n}}| v_{n}(x,0)|^{2} dx\geq 0
$$
with $ o_{n}(1)\rightarrow 0$. Then,
$$
\int_{\Upsilon_{n}}| v_{n}(x,0)|^2 dx \geq \frac{rM}{2R}-o_{n}(1)
$$
which contradicts (\ref{1contra}).
\end{proof}

\medskip

The other auxiliary functional we need is 
$I_{B_{\xi}}:H^1_{0,L}(\mathcal{C}_{B_{\xi}},y^{1-2\alpha})\rightarrow \mathbb{R}$, where $\xi>0$, given by
\begin{equation}\label{see}
I_{B_{\xi}}(v)=\frac{k_{\alpha}^{-1}}{2}\int_{\mathcal{C}_{B_{\xi}}} y^{1-2\alpha}|\nabla v|^{2} dxdy +
\frac{1}{2}\int_{B_{\xi}}|v(x,0)|^2 dx\   - \int_{B_{\xi}}
H(v(x,0)) dx.
\end{equation}
This functional has a Mountain Pass Geometry  and we denote with $c(B_{\xi})$  the mountain pass level.
If 
$$
\mathcal{M}_{B_{\xi}}=
\Big\{v \in H^1_{0,L}(\mathcal{C}_{B_{\xi}},y^{1-2\alpha})\backslash \{0\}: I'_{B_{\xi}}(v)[v]=0 \Big\}
$$
denotes the Nehari manifold associated to $I_{B_{\xi}}$, then, as usual,
\begin{equation}\label{blambda}
c(B_{\xi})= \inf_{v\in \mathcal{M}_{B_{\xi}}}I_{B_{\xi}}(v).
\end{equation}
Arguing as in  Theorem \ref{existence} and using Schwartz symmetrization techniques, we get
\begin{proposition}\label{4.4}
The functional $I_{B_{\xi}}$ defined in \eqref{see} satisfies the (PS) condition on $\mathcal M_{B_{\xi}}$.
In particular there exists a ground state solution
$ \mathfrak  w_{B_{\xi}}\in \mathcal M_{B_{\xi}} $ and 
$\mathfrak w_{B_{\xi}}(\cdot, 0)$ is radially symmetric
with respect to the origin.
\end{proposition}

%
%
%
\begin{proposition}\label{4.2}
The numbers $c(\Omega_{\lambda})$ and $c(B_{\xi})$, defined respectively in \eqref{clambda} and  \eqref{blambda},  verify the limits
$$
\lim_{\lambda
\rightarrow\infty}c(\Omega_{\lambda})=c(\mathbb R^N)\quad\mbox{and}\quad
\lim_{\xi \rightarrow\infty}c(B_{\xi})=c(\mathbb R^N).
$$
\end{proposition}
\begin{proof}
Here we will just prove the first limit, since the second one
follows from the same kind of arguments. Let $\Phi$ be a function
in $C^{\infty}(\mathcal{C}_{\mathbb R^{N}}; [0,1])$ such that 
$$
\Phi(x,y)=
\begin{cases}
1 & \mbox{ if }(x,y) \in \mathcal{C}_{B_{1}}\\
0 & \mbox{ if }(x,y) \in  \mathcal{C}_{\mathbb R^{N} \setminus B_{2}}. 
\end{cases}
$$
 For each $R>0$, let us
consider the rescaled function $\Phi_{R}(x,y)=\Phi(x/R,y)$ and set
$w_{R}(x,y)=\Phi_{R}(x,y)\mathfrak w_{\infty}(x,y)$, where $\mathfrak w_{\infty}$ is the ground state of the limit problem
given in Lemma \ref{gio16}, hence
$I_\infty(\mathfrak w_{\infty})=c(\mathbb R^N)$ and $I_\infty'(\mathfrak w_{\infty}) = 0$. 
Observe that 
\begin{equation}\label{contra}
w_{R}\to \mathfrak w_{\infty}\ \ \text{ in }\ \ H^{1}_{0,L}(\mathcal C_{\mathbb R^{N}}, y^{1-2\alpha})\ \text{ as } \ R\to +\infty.
\end{equation}
Since $0\in\Omega_{\lambda}$, there exists
$\bar\lambda>0$ such that $B_{2R}\subset \Omega_{\lambda}$ for
$\lambda\geq \bar\lambda$. Let $t_{R}>0$ such that
$$
I_{\lambda}(t_{R}w_{R})=\displaystyle\max_{t\geq 0}I_{\lambda}(t w_{R})=\displaystyle\max_{t\geq 0}I_{\infty}(t w_{R}).
$$
Thus $ I'_{\lambda}(t_{R}w_{R})[t_Rw_{R}]
= 0$, i.e. $t_{R}w_{R}\in \mathcal{M}_{\lambda}$.
Then
$$
c(\Omega_{\lambda})\leq I_{\lambda}(t_{R}w_{R})=I_{\infty}(t_{R}w_{R}) \;\
\mbox{for all} \;\ \lambda\geq \bar \lambda.
$$
Since $R$ is independent on $\lambda$, so is  $t_{R}$.
Hence, by taking the limit when $\lambda \rightarrow \infty$, we obtain
\begin{equation}\label{recalling}
 \displaystyle\limsup_{\lambda\rightarrow \infty}c(\Omega_{\lambda})\leq I_{\infty}(t_{R}w_{R}).
\end{equation}
{\sl Claim:} we have 
$\lim_{R\rightarrow \infty}t_{R}=1.$

Since $t_{R} w_{R}\in \mathcal M_{\lambda},$ 
we get
\begin{eqnarray*}\label{bo}
\|w_{R}\|_{\alpha}^{2}&=&\kappa_{\alpha}^{-1}\int_{\mathcal{C}_{\mathbb R^{N}} }y^{1-2\alpha}|\nabla w_{R}|^{2} dxdy +\int_{\mathbb{R}^N} |w_{R}(x,0)|^{2} dx \\
& = &\int_{\mathbb{R}^N}h(t_{R}w_{R}(x,0))t_{R}^{-1}w_{R}(x,0) dx \\
&>& \int_{B_{1}}h(t_{R}m)t_{R}^{-1}m dx,
\end{eqnarray*}
where $m=\min_{|x|\leq 1}w_{R}(x,0) > 0$ by the Strong Maximum Principle (see \cite[Remark 4.2]{CabreSire}).
It follows that $\{t_{R}\}$ has to be bounded, otherwise  by 
\eqref{H_{4}}  we deduce $\|w_{R}\|_{\alpha}^{2}\to+\infty$, against \eqref{contra}.

Moreover, if there exists 
$R_{n}\rightarrow \infty$ with $t_{R_{n}}\rightarrow 0$, by \eqref{H_{1}} and \eqref{H_{2}}
\begin{eqnarray*}
\|w_{R_n}\|_{\mathcal{C}_{\mathbb R^{N}}}^2 & = & \int_{\mathbb{R}^N}h(t_{R_n}w_{R_n}(x,0))t_{R_n}^{-1}w_{R}(x,0) dx\\
& \leq & C_1 t_{R_{n}}\int_{\mathbb{R}^N}|w_{R_{n}}(x,0)|^{2} dx + C_2 t_{R_n}^{q-1}\int_{\mathbb{R}^N}|w_{R_{n}}(x,0)|^q dx\to 0
\end{eqnarray*}
which again  contradicts \eqref{contra}. This implies that  $t_{R}\nrightarrow 0$. 
Thus, we can assume that  $t_{R}\rightarrow t_{0}>0$ for $R\to +\infty$ and consequently
$$
\kappa_{\alpha}^{-1}\int_{\mathcal{C}_{\mathbb R^{N}} }y^{1-2\alpha}|\nabla \mathfrak w_{\infty}|^{2} dxdy +\int_{\mathbb{R}^N} |\mathfrak w_{\infty}(x,0)|^{2} dx 
=\int_{\mathbb{R}^N}h(t_{0}\mathfrak w_{\infty}(x,0))t_{0}^{-1}\mathfrak w_{\infty}(x,0) dx.
$$
Since $ \mathfrak w_{\infty} \in {\cal{M}}_{\infty}$, it has to be $t_{0}=1$, proving our claim.

\smallskip

 Then
$I_{\infty}(t_{R}w_{R})\rightarrow I_{\infty}(\mathfrak w_{\infty})=c(\mathbb R^N)$ as
$R\rightarrow \infty$ and  recalling \eqref{recalling},
\begin{equation}\label{lsup}
\displaystyle\limsup_{\lambda \rightarrow \infty}c(\Omega_{\lambda})\leq c(\mathbb R^N). 
\end{equation}
On the other hand, by the definition of $c(\Omega_{\lambda})$ and $c(\mathbb R^N)$, we get
$ c(\Omega_{\lambda})\geq c(\mathbb R^N) \ \mbox{for all} \ \lambda>0, $
which implies 
\begin{equation}\label{linf}
\liminf_{\lambda \rightarrow \infty}c(\Omega_{\lambda})\geq c(\mathbb R^N).
\end{equation}
The conclusion follows by \eqref{lsup} and \eqref{linf}.
\end{proof}

\medskip

Before to proceed,  we need to introduce other notations. Given $a\in (-\infty,+\infty]$, we set \medskip
\begin{itemize}
\item $I_{\lambda}^{a}:=\Big\{u\in H^{1}_{0,L}(\mathcal C_{\Omega_{\lambda}}, y^{1-2\alpha}): I_{\lambda}(u)\leq a\Big\}$,
 the $a-$sublevel of $I_{\lambda}$; \medskip
\item $\mathcal M_{\lambda}^{a}:= \mathcal M_{\lambda}\cap  I_{\lambda}^{a}$. \medskip
\end{itemize}
Moreover, from now on we fix a
real number $r>0$ such that the sets 
$$
\Omega_{\lambda}^{+}=\{x\in \mathbb{R}^N; d(x,{\Omega_{\lambda}})\leq r\}
$$
and
$$
\Omega_{\lambda}^{-}=\{x\in \Omega; d(x,\partial\Omega_{\lambda})\geq r\}
$$
are homotopically equivalent to $\overline \Omega_{\lambda}$ and  $B_{\lambda r}\subset \Omega_{\lambda}$,
so that $ \mathcal M_{\lambda}^{c(B_{\lambda r})}\neq \emptyset.$

The next proposition will be of primary importance in order to apply the ``barycenter method''.
\begin{proposition}\label{fundamental}
There exists ${\lambda}^{*}>0$ such that for all $\lambda\geq\lambda^{*}$,
$$v\in \mathcal M_{\lambda}^{c(B_{\lambda r})}\   \Longrightarrow \beta(v)\in\Omega_{\lambda}^{+}.$$
\end{proposition}
\begin{proof}
Suppose that there exist $\lambda_{n}\rightarrow \infty$, $v_{n}\in
\mathcal{M}_{\lambda_{n}}^{c(B_{\lambda_{n} r})}$,
that we may assume positive, such that
$$
x_{n}=\beta(v^{+}_{n}(\cdot,0))\notin \Omega_{\lambda_{n}}^{+}.
$$
%
%
Fixing $R_{n}>\textrm{diam}(\Omega_{\lambda_{n}})$, we have that
$$
A_{\lambda_{n}R,\lambda_{n}r}(x_{n})\supset \Omega_{\lambda_{n}}
$$
and so, recalling \eqref{chap1}-\eqref{chap3},
\begin{equation}\label{4.6}
a(R,r,\lambda_{n})=a(R,r,\lambda_{n},x_{n})\leq I_{\lambda_{n}}(v_{n}) \leq c(B_{\lambda_{n}r}).
\end{equation}
Sending $n\rightarrow \infty$ in (\ref{4.6}) and using Proposition \ref{4.2}, it follows that \begin{eqnarray*}
\displaystyle\limsup_{n\rightarrow\infty}a(R,r,\lambda_{n})\leq c(\mathbb R^N)
\end{eqnarray*}
which contradicts Proposition \ref{4.1}.
\end{proof}

\medskip

 For $\lambda>0$, we define the injective operator
$\Psi_{\lambda, r}:\Omega_{\lambda}^{-}\rightarrow H^1_{0,L}(\mathcal{C}_{\Omega_{\lambda}},y^{1-2\alpha})$ 
given, for every $\tilde x\in \Omega_{\lambda}^{-}$ by
$$
[\Psi_{\lambda, r}(\tilde x)](x,y)=
\begin{cases}
t_{\lambda}\mathfrak w_{B_{\lambda r}}(|\tilde x-x|, y) & \ \mbox{for} \ (x,y)\in \mathcal C_{B_{\lambda r}(\tilde x)} \\
0& \ \mbox{for} \ (x,y)\in \mathcal C_{ \Omega_{\lambda}\setminus B_{\lambda r}(\tilde x)}
\end{cases}
$$
where $\mathfrak w_{B_{\lambda r}}$ is the ground state
solution given in Proposition \ref{4.4} and $t_{\lambda}>0$ is such that $\Psi_{\lambda, r}(\tilde x)\in \mathcal M_{\lambda}$, see \eqref{mappanehari}.
Note that for every $\tilde x\in \Omega_{\lambda}^{-}$, it holds
$$
\beta(\Psi_{\lambda, r}(\tilde x))=\beta([\Psi_{\lambda, r}(\tilde x)](\cdot,0))=\tilde x
$$
and since $$I_{\lambda}(\Psi_{\lambda, r}(\tilde x))=I_{B_{\lambda r}}(t_{\lambda}\mathfrak w_{B_{\lambda r}}(|\tilde x-\cdot|, \cdot))\leq
I_{B_{\lambda r}}( \mathfrak w_{B_{\lambda r}}(|\tilde x-\cdot|, \cdot))=
c(B_{\lambda r}),$$
we infer also
$$
\Psi_{\lambda, r}(\tilde x)\in \mathcal M_{\lambda}^{c(B_{\lambda r})}.
$$
%
Then we have
\begin{lemma}\label{homotopia} 
For $\lambda\geq  \lambda^{*}$ given in Proposition \ref{fundamental}, the composite map
$$\Omega_{\lambda}^{-} \stackrel{\Psi_{\lambda , r}}{\longrightarrow}
 \mathcal M_{\lambda}^{c(B_{\lambda r})} \stackrel{ \beta}{\longrightarrow}\Omega_{\lambda}^{+}$$
is well defined and coincide with the inclusion map of $\Omega_{\lambda}^{-}$ into $\Omega_{\lambda}^{+}$
\end{lemma}
The next result is a consequence of the above setting, but for the sake of completeness
we give the proof. It is understood, from now on, that for $\lambda^{*}$ we mean that given in Proposition \ref{fundamental}.
\begin{proposition}\label{4.5}
For  every $\lambda \geq {\lambda}^{*}$ we have
$$
cat \, \mathcal M_{\lambda}^{c(B_{\lambda r})}\geq
cat \,\Omega_{\lambda}.
$$
\end{proposition}
\begin{proof}
Assume that $cat \,\mathcal M_{\lambda}^{c(B_{\lambda r})}=n$. This means that
$n$ is the smallest positive integer such that
$$
\mathcal M_{\lambda}^{c(B_{ \lambda r})}=\bigcup _{j=1}^{n}A_{j},
$$
where $A_{j}, j=1,\ldots,n$ are closed and contractible in $\mathcal M_{\lambda}^{c(B_{\lambda r})}$;
that is, there exist $h_{j}\in C([0,1]\times A_{j},\mathcal M_{\lambda}^{c(B_{\lambda r})})$
and fixed elements $w_j \in \mathcal M_{\lambda}^{c(B_{\lambda r})}$ 
 such that
$$
h_{j}(0,u)=u \,\, \ \textrm{for all }u \in A_j \qquad \mbox{and} \qquad
h_{j}(1,u)=w_j \ \mbox{for all} \ u\in A_{j}.
$$
 Consider the closed sets $D_{j}=\Psi_{\lambda, r}^{-1}(A_{j})$ and note that 
$$
\Omega_{\lambda}^{-}= \bigcup_{j=1}^{n}D_{j}.
$$
Using the deformation $g_{j}:[0,1]\times D_{j}\rightarrow
 \Omega_{\lambda}^{+}$ given by
$$
g_{j}(t,x)=\beta\Big((h_{j}(t, \Psi_{\lambda, r}(x))^{+}(\cdot,0)\Big),
$$
we have for $j=1,\ldots, n$ and $x\in D_{j}$
$$
g_{j}(0,x)=\beta\Big((h_{j}(0, \Psi_{r}(x)))^{+}(\cdot,0)\Big)= \beta\Big(\Psi_{\lambda,r}(x)(\cdot,0)\Big)=x
$$
and
$$
g_{j}(1,x)=\beta\Big((h_{j}(1, \Psi_{r}(z)))^{+}(\cdot,0)\Big)= \beta\Big(w_j(\cdot,0)^{+}\Big)\in \Omega_{\lambda}^{+}.
$$
 This means that  $D_{j}, j=1,\ldots,n$ is contractible 
in $\Omega_{\lambda}^{+}$, hence
$cat_{\Omega_{\lambda}^{+}}(\Omega_{\lambda}^{-})\leq n$. 
The conclusion follows
since $\Omega_{\lambda}^{+}$ and $\Omega^{-}_{\lambda}$ are homotopically equivalent to $\overline \Omega_{\lambda}$.
\end{proof}

\section{Proof of Theorem \ref{Main2} and Theorem \ref{+1}}\label{1}
Let us fix $\lambda\geq\lambda^{*}$. 
Since $I_{\lambda}$ satisfies the Palais-Smale condition on
 $\mathcal{M}_{\lambda}$, applying the Ljusternik-Schnirelmann theory and Proposition \ref{4.5}, we get $I_{\lambda}$ on
  $\mathcal{M}_{\lambda}$ has at least $cat \,\Omega_{\lambda}$ critical points whose energy
   is less than $c(B_{\lambda r})$. Moreover, all solutions obtained are positive
    by the maximum principle proved in Lemma \ref{maximumprinciple}, finishing the proof of Theorem \ref{Main2}.
    
To get another solution, and then proving Theorem \ref{+1}, we use the same ideas of \cite{BCP}.
Since $\Omega_{\lambda}$ is not contractible, 
the compact set  $A:=\overline{{\Phi_{\lambda,r}(\Omega_{\lambda}^{-})}}$
can not be contractible in $\mathcal M_{\lambda}^{c(B_{\lambda r})}$.
Moreover, as in \cite{BC3}, one can show  that functions on the Nehari manifold have to be positive
on a set of nonzero measure.

In the following, for $u\in H^{1}_{0,L}(\Omega_{\lambda}, y^{1-2\alpha})\setminus\{0\}$
we denote with $t_\lambda (u)>0$ the unique positive number such that $t_{\lambda}(u) u\in \mathcal M_{\lambda}.$

\medskip

Take $u^{*}\in H^{1}_{0,L}(\Omega_{\lambda}, y^{1-2\alpha}) $  such that $u^{*}\geq 0$, and 
$I_{\lambda}(t_{\lambda}(u^{*})u^{*})>c(B_{\lambda r}).$
Consider the cone
$$\mathcal K:=\Big\{tu^{*}+(1-t)u: t\in [0,1], u\in A \Big\}$$
(which is  compact and contractible) and, since  functions in $\mathcal K$ 
have to be positive on a set of nonzero measure,  $0\notin \mathcal K$.
Then it makes sense to project the cone on the Nehari manifold
$$t_{\lambda}(\mathcal K):=\Big\{t_{\lambda}(w)w: w\in \mathcal K\Big\}\subset \mathcal M_{\lambda}$$
and consider the number
$$c:=\max_{t_{\lambda}(\mathcal K)}I_{\lambda}>c(B_{\lambda r}).$$
Since $ A\subset t_{\lambda}(\mathcal K)\subset \mathcal M_{\lambda}$
and $t_{\lambda}(\mathcal K)$ is contractible in $\mathcal M_{\lambda}^{c}$,
we infer that also $ A$ is contractible in $\mathcal M^{c}_{\lambda}$.
In conclusion, $ A$  is contractible in $\mathcal M^{c}_{\lambda}$,
 not contractible  in $\mathcal M^{c(B_{\lambda r})}_{\lambda}$, and  $c>c(B_{\lambda r});$
this is only possible, since $I_{\lambda}$ satisfies the Palais-Smale condition, if there is
a critical level  between $c(B_{\lambda r})$ and $c$, that is, another solution to our problem.

\section{Proof of Theorem \ref{Morse}}\label{2}

Before prove the theorem we  recall some basic facts of Morse theory
and fix some notations.
For a pair of topological spaces $(X,Y)$,
$Y\subset X,$ let $H_{*}(X,Y)$ be its singular homology with coefficients in some field $\mathbb F$
(from now on omitted) and 
$$\mathcal P_{t}(X,Y)=\sum_{k}\dim H_{k}(X,Y)t^{k}$$
the Poincar\'e polynomial of the pair. If $Y=\emptyset$, it will be always omitted in the objects which involve the pair.
Recall that  if $H$ is an Hilbert space,  $I:H\to \mathbb R$  a $C^{2}$ functional  and 
$u$ an isolated critical point  with $I(u)=c$, the  {\sl polynomial Morse index} of $u$ is
$$\mathcal I_{t}(u)=\sum_{k}\dim C_{k}(I,u)  t^{k}$$
where $C_{k}(I,u)=H_{k}(I^{c}\cap U, (I^{c}\setminus\{u\})\cap U)$ are the critical groups.
Here $I^{c}=\{u\in H: I(u)\leq c\}$ and $U$ is a neighborhood of the critical point $u$.
The multiplicity of $u$ is the number $\mathcal I_{1}(u)$.

It is known that  for a non-degenerate critical point $u$
(that is, the selfadjoint operator associated to $I''(u)$
is an isomorphism)
it is $\mathcal I_{t}(u)=t^{\mathfrak m (u)}$,
where $\mathfrak m(u)$ is the {\sl (numerical) Morse index of $u$}: the maximal dimension
of the subspaces where $I''(u)[\cdot,\cdot]$ is negative definite.

\medskip
Coming back to our functional,  we know that  $I_{\lambda}$ satisfies the Palais-Smale condition (see Lemma \ref{lema_PS_unconstrained}).
Moreover $I_{\lambda}$ is of class $C^{2}$ and
for $v,v_{1},v_{2}\in H^{1}_{0,L}(\mathcal C_{\Omega_{\lambda}}, y^{1-2\alpha})$ it is
\begin{multline*}
I_{\lambda}''(v)[v_{1},v_{2}]=k_\alpha^{-1} 
\int_{\mathcal{C}_{\Omega_{\lambda}}}y^{1-2\alpha}\nabla v_{1} \nabla v_{2}\, dx dy + \\ 
\int_{\Omega_{\lambda}} v(x,0) w(x,0) dx - \int_{\Omega_{\lambda} } h'(v(x,0))v_{1}(x,0) v_{2}(x,0)dx.
\end{multline*}
So  $I_{\lambda}''(v)$ is represented by the operator
\begin{equation}\label{Lv}
\mathrm L_{\lambda}(v):=\mathrm R_{\lambda}(v)-\mathrm K_{\lambda}(v):
H^{1}_{0,L}(\mathcal C_{\Omega_{\lambda}}, y^{1-2\alpha})\to \Big(H^{1}_{0,L}(\mathcal C_{\Omega_{\lambda}}, y^{1-2\alpha}) \Big)' 
\end{equation}
where $\mathrm R_{\lambda}(v)$ is the Riesz isomorphism and $\mathrm K_{\lambda}(v)$ is compact. Indeed let
 $v_{n}\rightharpoonup 0$ in $H^{1}_{0,L}(\mathcal C_{\Omega_{\lambda}}, y^{1-2\alpha})$ and 
 $w\in H^{1}_{0,L}(\mathcal C_{\Omega_{\lambda}}, y^{1-2\alpha})$;
in virtue of \eqref{H_{1}'} and \eqref{H_{2}'}, for a given  $\xi > 0$ there exists some constant $C_{\xi}>0$ such that 
$$\int_{\Omega_{\lambda}}\Big|h'(v(x,0))v_{n}(x,0)w(x,0)\Big|dx\leq \xi \int_{\Omega_{\lambda}} |v_{n}(x,0)w(x,0) |dx 
+C_{\xi}\int_{\Omega_{\lambda}}|v(x,0)|^{q-1}|v_{n}(x,0)w(x,0)|dx.$$
Using that  $v_{n}\rightharpoonup 0$ and the arbitrariness of  $\xi$, we get
$$\|\mathrm K_{\lambda}(v)[v_{n}]\|=\sup_{\|w\|_\alpha=1}\Big|\int_{\Omega_{\lambda}}h'(v(x,0))v_{n}(x,0)w(x,0)dx\Big|\rightarrow 0.$$
In particular $\mathrm L_{\lambda}(v)$ is a Fredholm operator with index zero.
Moreover, for $a\in(-\infty,+\infty]$, we set
\begin{itemize}
\item $\textrm{Crit}_{\lambda}:=\Big\{u\in H^{1}_{0,L}(\mathcal C_{\Omega_{\lambda}}, y^{1-2\alpha}): I'_{\lambda}(u)=0\Big\}$, 
 the set of critical points of $I_{\lambda}$; \medskip
\item$(\textrm{Crit}_{\lambda})^{a}:= \textrm{Crit}_{\lambda}\cap  I_{\lambda}^{a}$; \medskip 
\item $ (\textrm{Crit}_{\lambda})_{a}:=\Big\{u\in \textrm{Crit}_{\lambda}: I_{\lambda}(u)> a\Big\} $.
\end{itemize}

In the remaining part of this section we will follow \cite{CSR,BC3}. We will
not give the the proofs of the next Lemma \ref{Multiplic} and Corollary \ref{Quadrato}
since they follows by general arguments.

Let $\lambda^{*}>0$ as given in Proposition \ref{fundamental}
and  $\lambda\geq \lambda^{*}$ be {\sl fixed} from now on.
In view of  Corollary \ref{17}, to prove Theorem \ref{Morse} it is sufficient to show that
 $I_{\lambda}$ restricted to $\mathcal M_{\lambda}$
 has at least $2\mathcal P_{1}(\Omega_{\lambda})-1$ critical points.
 
 \medskip
 
First note that we can assume that $c(B_{\lambda r})$ is a regular value for $I_{\lambda}$. Otherwise
we can choose a $\rho\in (0,r)$ so that the new sets
$$
\Omega_{\lambda}^{+}=\{x\in \mathbb{R}^N; d(x,{\Omega_{\lambda}})\leq \rho\}
\ \ \text{ and }\ \ \Omega_{\lambda}^{-}=\{x\in \Omega; d(x,\partial\Omega_{\lambda})\geq \rho\}
$$
are still homotopically equivalent to $\Omega$, 
$c(B_{\lambda \rho})>c(B_{\lambda r})$ and 
$c(B_{\lambda \rho})$ is a regular value;
and we rename $c(B_{\lambda \rho})$ as $c(B_{\lambda r})$.
Of course, we can  also assume   
 that $\textrm{Crit}_{\lambda}$ is discrete. Since $I_{\lambda}$ is bounded from below on $\mathcal M_{\lambda}$,
 let us say by a $\delta_{\lambda}>0$,
 we have
 $$(\textrm{Crit}_{\lambda})^{c(B_{\lambda r})}=
 \Big\{v\in \textrm{Crit}_{\lambda}: 0<\delta_{\lambda}<I_{\lambda}(v)\leq c(B_{\lambda r})\Big\}$$
and $(\textrm{Crit}_{\lambda})^{c(B_{\lambda r})}$ and $(\textrm{Crit}_{\lambda})_{c(B_{\lambda r})}$
are (critical) isolated sets covering $\textrm{Crit}_{\lambda}$.

By Lemma \ref{homotopia} and the fact that $(\Psi_{\lambda, r})_{*}$ induces monomorphism between the homology groups
$H_{*}(\Omega_{\lambda}^{-})$ and $H_{*}(\mathcal M_{\lambda}^{c(B_{\lambda r})})$, it is standard to see that 
\begin{equation}\label{injecoes}
 \mathcal P_{t}(\mathcal M_{\lambda}^{c(B_{\lambda r})})= \mathcal P_{t}(\Omega_{\lambda}^{-}) +\mathcal Q_{t}, 
 \quad \mathcal Q\in \mathbb P 
\end{equation}
where we are denoting with $\mathbb P$ the set of polynomial with nonnegative integer coefficients.
Recall that $c(\Omega_{\lambda})=\min_{\mathcal M_{\lambda}}I_{\lambda}.$ As in \cite[Lemma 5.2]{BC3}
(the proof just uses a topological lemma and a general deformation argument) one proves the following 
\begin{lemma}\label{Multiplic}
Let $d\in (0,c(\Omega_{\lambda}))$ and $\l\in (d,+\infty]$ a regular level for $I_{\lambda}.$
Then \begin{equation*}\label{multiplic}
\mathcal P_{t}(I_{\lambda}^{\l}, I_{\lambda}^{d})=t \mathcal P_{t}(\mathcal M_{\lambda}^{\l}).
\end{equation*}
\end{lemma}
From this lemma, \eqref{injecoes} and the fact that $\pi_{1}(\mathcal M_{\lambda})\approx\{0\}$, it follows that
\begin{equation}\label{uno}
\mathcal P_{t}(I_{\lambda}^{c(B_{\lambda r})}, I_{\lambda}^{d})= t \Big(\mathcal P_{t}(\Omega_{\lambda}^{-})+\mathcal Q_{t}\Big)
\end{equation}
and
\begin{equation}\label{due}
 \mathcal P_{t}(H^{1}_{0,L}(y^{1-2\alpha}), I_{\lambda}^{d})= t \mathcal P_{t}( \mathcal M_{\lambda})=t.
\end{equation}
Finally we need the next result, whose proof is a  matter of algebraic topology (see \cite[Lemma 2.4]{CSR} or  \cite[Lemma 5.6]{BC3})
\begin{corollary}\label{Quadrato}
We have
\begin{equation}\label{quadrato}
 \mathcal P_{t}(H^{1}_{0,L}(y^{1-2\alpha}), I_{\lambda}^{c(B_{\lambda r})})
 =t^{2}\Big(\mathcal P_{t}(\Omega_{\lambda}) +\mathcal Q_{t}-1\Big), \quad \mathcal Q\in \mathbb P.
\end{equation}
\end{corollary}

Then the Morse theory,  \eqref{uno}, \eqref{due} and \eqref{quadrato} give
\begin{eqnarray*}
\sum_{v\in (\textrm{Crtit}_{\lambda})^{c(B_{\lambda r})}}\mathcal I_{t}(v)&=&
 \mathcal P_{t}(I_{\lambda}^{c(B_{\lambda r})}, I_{\lambda}^{d}) +(1+t)\mathcal Q'_{t} \\
&=& t\Big(\mathcal P_{t}(\Omega_{\lambda})+\mathcal Q_{t}\Big)+(1+t)\mathcal Q'_{t} 
\end{eqnarray*}
and
\begin{eqnarray*}
\sum_{v\in (\textrm{Crtit}_{\lambda})_{c(B_{\lambda r})}}\mathcal I_{t}(v)&=&
 \mathcal P_{t}(H^{1}_{0,L}(y^{1-2\alpha}), I_{\lambda}^{c(B_{\lambda r})}) +(1+t)\mathcal Q''_{t} \\
&=& t^{2}\Big(\mathcal P_{t}(\Omega_{\lambda})+\mathcal Q_{t}-1\Big)+(1+t)\mathcal Q''_{t} 
\end{eqnarray*}
for some $\mathcal Q, \mathcal Q', \mathcal Q'' \in \mathbb P.$
As a consequence we obtain
\begin{equation}\label{finale}
 \sum_{v\in \textrm{Crtit}_{\lambda}}\mathcal I_{t}(v)=
 t\mathcal P_{t}(\Omega_{\lambda})
+t^{2}\Big(\mathcal P_{t}(\Omega_{\lambda})-1\Big)+t(1+t)\mathcal Q_{t}
\end{equation}
for a suitable $\mathcal Q\in \mathbb P$.

It is known that  for a non-degenerate critical point $v$
(that is, $\mathrm L_{\lambda}(v)$ given in \eqref{Lv}
is an isomorphism)
it is $\mathcal I_{t}(v)=t^{\mathfrak m (v)}$,
where $\mathfrak m(v)$ is the {\sl (numerical) Morse index of $v$}: the maximal dimension
of the subspaces where $I_{\lambda}''(v)[\cdot,\cdot]$ is negative definite.

Then, if the solutions are non-degenerate, \eqref{finale} easily gives the existence of at least $2\mathcal P_{1}(\Omega_{\lambda})-1$
solutions, completing the proof of Theorem \ref{Morse}.

\bigskip

\noindent \textbf{Acknowledgment.} The first author would like to thank UNESP - Presidente Prudente, specially to
 Professor Marcos Pimenta for his attention and friendship. This work was done while he was visiting that institution.

\end{document}